\documentclass[11pt]{amsart}

\usepackage{amsfonts}
\usepackage{amssymb}
\usepackage{amsthm}
\usepackage{amsmath}
\usepackage{hyperref}
\usepackage{enumitem}
\usepackage{tikz}
\usepackage{pgf}
\usepackage{graphicx}
\usepackage[top=3.5cm, bottom=3.5cm, left=2.5cm, right=2.5cm]{geometry}

\newtheorem{prop}{Proposition}[section]
\newtheorem{theorem}[prop]{Theorem}
\newtheorem{lemma}[prop]{Lemma}

\newtheorem{conjecture}[prop]{Conjecture}

\theoremstyle{definition}

\theoremstyle{remark}

\newtheorem*{remark*}{Remark}
\newtheorem{remark}[prop]{Remark}

\theoremstyle{theorem}

\newcommand{\R}{\mathbb{R}}

\newcommand{\N}{\mathbb{N}}
\newcommand{\Z}{\mathbb{Z}}

\newcommand{\Aut}{\text{\textup{Aut}}\,}

\numberwithin{equation}{section}

\title{Horofunctions on graphs of linear growth}
\date{}
\author{Matthew C. H. Tointon}
\address{MT: Laboratoire de Math\'ematiques d'Orsay, Univ.~Paris-Sud, CNRS, Universit\'e Paris-Saclay, 91405 Orsay, France}
\email{matthew.tointon@math.u-psud.fr}

\author{Ariel Yadin}
\address{AY: Department of Mathematics, Ben-Gurion University of the Negev, Beer-Sheva, Israel}
\email{yadina@bgu.ac.il}

\begin{document}
\maketitle

\begin{abstract}
We prove that a linear growth graph has finitely many horofunctions.
This provides a short and simple proof that any finitely generated infinite group of linear growth
is virtually cyclic.
\end{abstract}
\renewcommand{\abstractname}{R\'esum\'e}
\begin{abstract}
Nous montrons qu'un graphe \`a croissance lin\'eaire admet un nombre fini d'horofonctions. Cel\`a donne une preuve courte et simple que chaque groupe infini de type fini \`a croissance lin\'eaire est virtuellement cyclique.
\end{abstract}

\section{Introduction}

%

Let $\Gamma$ be an infinite, connected, locally finite graph,
and denote by $d(\cdot,\cdot)$ the graph metric on $\Gamma$.
Let $o\in\Gamma$. Then $\Gamma$ is said to have {\em linear volume growth} if the balls about $o$ for the metric $d$ grow at most linearly in the radius. The graph $\Gamma$ is said to have {\em polynomial volume growth} if these balls grow at most polynomially.

Given an element $z \in \Gamma$ we define the \emph{Busemann function} $b_z:\Gamma\to\Z$ via $b_z(y) = d(z,y) - d(z,o)$. Given a geodesic ray $\omega=(z_1,z_2,\ldots)$ in $\Gamma$ we define the \emph{horofunction} $f_\omega:G\to\Z$ by $f_\omega(y)=\lim_{n\to\infty}b_{z_n}(y)$. It is a well-known and simple fact that this limit exists. Note that $f_\omega$ is not constant, 
and in fact $f_\omega(z_n) = -n$, which shows that $f_\omega$ is unbounded.


\begin{theorem} \label{thm:lin.Buse}
Let $\Gamma$ be an infinite, connected, locally finite graph of linear volume growth.  
Then the set of horofunctions on $\Gamma$ is finite.
\end{theorem}

A finitely generated group $G$ is said to have polynomial (respectively, linear) volume growth 
if some (and hence every) Cayley graph of $G$ has polynomial (respectively, linear) volume growth. A remarkable theorem of Gromov's states that a finitely generated group of polynomial volume growth contains a nilpotent subgroup of finite index \cite{gromov}. As an application of Theorem \ref{thm:lin.Buse} we give a short argument to prove the linear-growth case of Gromov's theorem, as follows.
\begin{theorem}\label{thm:lin}
Let $G$ be a finitely generated infinite group of linear volume growth. 
Then $G$ contains a cyclic subgroup of finite index.
\end{theorem}
In fact, Gromov's theorem implies that a group of subquadratic growth is virtually cyclic, and this has also been proved by elementary methods by Justin \cite{justin}, van den Dries \& Wilkie \cite{vdd-wilkie} and Imrich \& Seifter \cite{imr-sei}, the last two of these giving bounds on the index of the cyclic subgroup in terms of the volume growth. Nonetheless, the present proof is completely different to all of those and rather short, so we record it here.

We prove Theorem \ref{thm:lin.Buse} in Section \ref{sec:horo}. The proof of Theorem \ref{thm:lin} is in Section \ref{sec:linear}.

Let us mention a related (probably much more difficult) question.
\begin{conjecture}
Let $G$ be a Cayley graph of polynomial volume growth.
Then the set of horofunctions on $G$ is countable.
\end{conjecture}

Proving this conjecture would provide an alternative proof to Gromov's theorem
(by using a variant of Lemma \ref{lem:fin.orbit} below).  This has been suggested by Karlsson \cite{anders}.
One method to prove this conjecture could be using the structure of finitely generated nilpotent groups, 
and relying on Gromov's theorem, but that would somehow miss the point.  (For example, in \cite{walsh} 
Walsh shows that nilpotent groups always have a finite orbit in the space of horofunctions.  It seems that 
this can be extended to virtually nilpotent groups as well.)
It would be interesting to prove this conjecture even in the quadratic growth case 
without using Gromov's theorem, since that would imply a new proof of the characterization of 
recurrent groups (which are finite extensions of $\Z$ or $\Z^2$).

\section{Horofunctions on a graph of linear growth}\label{sec:horo}

For a graph $\Gamma$, we say that $\Gamma$ is 
\emph{$\N$-partite on a sequence $(\Gamma_n)_{n\in \N}$} of disjoint sets $\Gamma_n$
if $\Gamma$ has vertex set $\bigcup_{n \in \N} \Gamma_n$ and the neighbours of every $x\in \Gamma_n$ lie in $\Gamma_{n-1}\cup \Gamma_{n+1}$. 
We call the sets $\Gamma_n$ the \emph{partite sets} of $\Gamma$. 
We call a path in $\Gamma$ \emph{monotone} if it has at most one vertex in each $\Gamma_n$. 

The proof of Theorem \ref{thm:lin.Buse} essentially rests on the following graph-theoretic result.
\begin{prop}\label{prop:graph}
Let $\Gamma$ be an $\N$-partite graph whose partite sets all have cardinality $k \geq 1$. 
Then there exist monotone paths $\gamma_1,\ldots,\gamma_k$ in $\Gamma$
such that every infinite monotone path in $\Gamma$ has infinite intersection with some $\gamma_j$.
\end{prop}

Recall that in a bipartite graph on two finite sets $X_1,X_2$ of equal cardinality, a \emph{matching} is a subgraph in which each element of $X_1$ is connected to precisely one element of $X_2$, and vice versa. Hall's Marriage Theorem \cite{hall} states that if there is no matching then there exists some subset $Y\subset X_1$ such that the neighbourhood of $Y$ in $X_2$ has strictly smaller cardinality than $Y$ itself.
\begin{lemma}\label{lem:match}
If for each $n\in\N$ there is a matching in $\Gamma$ between $\Gamma_n$ and 
$\Gamma_{n+1}$ then $\Gamma$ satisfies Proposition \ref{prop:graph}.
\end{lemma}
\begin{proof}
It is easy to see that the existence of such matchings implies that the vertices of $\Gamma$ may be partitioned into $k$ monotone paths, and that this is sufficient to satisfy the proposition.
\end{proof}


Given an $\N$-partite graph $\Gamma$ and a sequence $N = (n_j)_j \subset \N$, 
we may define
a new graph on $\Gamma_N : = \bigcup_{j} \Gamma_{n_j}$
by placing an edge between $x \in \Gamma_{n_j}$ and $x' \in\Gamma_{n_{j+1}}$ 
if and only if there exists a monotone path between $x$ and $x'$ in $\Gamma$. 
(Note that $\Gamma_N$ is an $\N$-partite graph with partite sets $\Gamma_{n_j}$.)
The following is then immediate.

\begin{lemma}\label{lem:ind}
If there exists a sequence $N = (n_j)_j$ such that $\Gamma_N$ satisfies 
Proposition \ref{prop:graph}, then the conclusion of Proposition \ref{prop:graph} 
holds for $\Gamma$ as well.
\end{lemma}

\begin{proof}[Proof of Proposition \ref{prop:graph}]
We proceed by induction on $k$. The base case $k=1$ is easy, so we assume that $k>1$. We may also delete every element of $\Gamma$ that does not lie in any infinite monotone path; the only potential problem with this is that the $\Gamma_n$ may no longer all have the same cardinality, but, using Lemma \ref{lem:ind}, we may fix this by passing to a subsequence.

Let $N = (n_j)_{j}$ be a sequence and consider the $\N$-partite graph $\Gamma_N$.
If for every $j$ there exists a matching (in the graph $\Gamma_N$) 
between $\Gamma_{n_j}$ and $\Gamma_{n_{j+1}}$,
then we are done by combining Lemmas \ref{lem:match} and \ref{lem:ind}.

Thus, we assume that a sequence as above does not exist. 
Specifically, by Hall's Marriage Theorem,
there exists $n$ such that for any $m>n$, 
there exist $U_m \subset \Gamma_n$ and $V_m \subset \Gamma_m$ 
such that
$1 \leq |V_m| < |U_m| \leq k$, and such that every monotone path from $U_m$ to $\Gamma_m$ ends in $V_m$. Without loss of generality we assume that the sets $V_m$ are minimal with respect to these properties, and hence that every element $v\in V_m$ lies in some monotone path from $U_m$ to $\Gamma_m$.
By passing to a subsequence of $m>n$, we have an infinite sequence 
$n<m_1 < m_2 < \cdots$ such that 
$|V_{m_j}| = |V_{m_1}|$ all have the same size and 
$U_{m_j} = U_{m_1} = U$ are all the same fixed subset.

Let $M = (m_j)_j$ and consider the graph $\Gamma_M$.
We claim that $\Gamma_M$ satisfies Proposition \ref{prop:graph}, which will suffice by Lemma \ref{lem:ind}.
We move to proving this claim.

Every monotone path in $\Gamma$ starting in $U$ and ending in $\Gamma_{m_j}$ must 
end in $V_{m_j}$.  Thus, by minimality of $V_{m_j}$, any monotone path in $\Gamma$ starting in $V_{m_j}$ and ending in 
$\Gamma_{m_{j+1}}$ must end in $V_{m_{j+1}}$.
In particular, in the graph $\Gamma_M$, 
any infinite monotone path $\gamma$ must satisfy the following dichotomy:
either $\gamma \cap V_{m_j} = \varnothing$ for all $j$, or
there exists $j_0$ such that for all $j>j_0$ we have $\gamma \cap V_{m_j} \neq \varnothing$.

Let $\Gamma_A$ be the induced subgraph of $\Gamma_M$ on the vertex set $\bigcup_j V_{m_j}$, and 
$\Gamma_B$ the induced subgraph on the vertex set $\bigcup_j (\Gamma_{m_j} \setminus V_{m_j})$.
Note that $\Gamma_A$ is $\N$-partite with partite sets $(\Gamma_A)_j = V_{m_j}$, 
all of size $v=|V_{m_1}| < k$. 
Also, $\Gamma_B$ is $\N$-partite with partite sets
$(\Gamma_B)_j = \Gamma_{m_j} \setminus V_{m_j}$, which all have size
$w=k - v < k$.
Thus, any infinite monotone path in $\Gamma_M$ induces an infinite monotone
path in either $\Gamma_A$ or in $\Gamma_B$.
By induction,
there exist infinite monotone paths $\alpha_1,\ldots, \alpha_v$ 
in $\Gamma_A$ such that any infinite monotone path in $\Gamma_A$ must intersect 
one of these infinitely many times. 
Similarly, there are such paths $\beta_1,\ldots, \beta_w$ in $\Gamma_B$.
Thus, any infinite monotone path in $\Gamma_M$ must intersect one of
$\alpha_1,\ldots, \alpha_v , \beta_1 , \ldots, \beta_w$ infinitely many times.
This completes the proof.
\end{proof}

\begin{proof}[Proof of Theorem \ref{thm:lin.Buse}]
Note that if $\Gamma$ has linear growth then, writing $B_r$ for the ball of radius $r$ about $o$, there is some $k\in\N$ and an infinite increasing sequence $m_1,m_2,\ldots$ such that $|B_{m_n}\backslash B_{m_n-1}|=k$ for every $n$. Define $\Gamma_n=B_{m_n}\backslash B_{m_n-1}$, and define an $\N$-partite graph $\hat\Gamma$ on $\Gamma_1,\Gamma_2,\ldots$ by joining $x\in \Gamma_n$ to $x'\in \Gamma_{n+1}$ if and only if there is a path in $\Gamma$ from $x$ to $x'$ of length $m_{n+1}-m_n$.

Define a map $\alpha$ from the set of geodesic rays in $\Gamma$ starting at $o$ to the set of monotone paths in $\hat\Gamma$ (in the sense of Proposition \ref{prop:graph}) in the natural way. Specifically, if $\omega$ is a geodesic ray in $\Gamma$ starting at $o$ then $\alpha(\omega)$ is the unique monotone path in $\hat\Gamma$ passing through the same elements of $\bigcup_n\Gamma_n$ as $\omega$. Note that $\alpha$ is surjective onto the set of monotone paths in $\hat\Gamma$, and also that if $\alpha(\omega)$ and $\alpha(\omega')$ have infinite intersection then so do $\omega$ and $\omega'$.

Let $\gamma_1,\ldots,\gamma_k$ be as given by Proposition \ref{prop:graph}, and pick, using the surjectivity of $\alpha$, geodesic rays $\omega_1,\ldots,\omega_k$ in $\Gamma$ starting at $o$ such that $\alpha(\omega_i)=\gamma_i$. If $\beta$ is a geodesic ray in $\Gamma$, the tail of $\beta$ coincides with the tail of some geodesic ray $\beta'$ in $\Gamma$ starting at $o$ (see Lemma \ref{lem:ray.from.1} below). However, $\alpha(\beta')$ has infinite intersection with some $\gamma_i$ by Proposition \ref{prop:graph}, and so $\beta$ has infinite intersection with $\omega_i$. This implies in particular that $f_\beta=f_{\omega_i}$, and so $f_{\omega_1},\ldots,f_{\omega_k}$ is a complete set of horofunctions.
\end{proof}

For completeness we include a short argument for the following standard lemma.

\begin{lemma}\label{lem:ray.from.1}
If $\gamma=(x_0,x_1,x_2,\ldots)$ is a geodesic ray starting at $x_0$ then there exists some $N$ such that $(x_N,x_{N+1},\ldots)$ coincides with the tail of a geodesic ray $\omega$ starting at $o$.
\end{lemma}
\begin{proof}
The sequence $d(x_n,o)-d(x_n,x_0)$ is non-increasing in $n$, since
\begin{equation}\label{eq:ray}
d(x_{n+1},x_0)=d(x_n,x_0)+1
\end{equation}
and $|d(x_{n+1},o)-d(x_n,o)|\le1$ for every $n$. The triangle inequality also implies that $d(x_n,o)-d(x_n,x_0)$ is bounded below by $-d(o,x_0)$. The sequence $(d(x_n,o)-d(x_n,x_0))_{n=1}^\infty$ is therefore eventually constant, say for $n\ge N$. Combined with \eqref{eq:ray}, this implies that $d(x_{n+1},o)=d(x_n,o)+1$ for $n\ge N$. The infinite path $\omega$ having initial segment some geodesic path from $o$ to $x_N$, followed by $x_{N+1},x_{N+2}, \ldots$, is therefore a geodesic ray starting at $o$.
\end{proof}



\section{The linear-growth case of Gromov's theorem}\label{sec:linear}

A group $G$ acts on the space $\{f:G\to\R\,|\,f(1)=0\}$ by $x\cdot f(y)=f(x^{-1}y)-f(x^{-1})$. Note that for a Busemann function $b_z$ we have $x\cdot b_z=b_{xz}$, and hence for a horofunction $f_\omega$ we have $x\cdot f_\omega=f_{x\omega}$.

The following observation we learned from Anders Karlsson.

\begin{lemma} \label{lem:fin.orbit}
If the set of horofunctions on a group $G$ contains a finite orbit then $G$ has a finite-index subgroup admitting a surjective homomorphism onto $\Z$.
\end{lemma}

\begin{proof}
Letting $G$ act on the finite orbit, $G$ contains a finite-index subgroup $H$ that fixes some element $f_\omega$ of the orbit. Thus for $h\in H, g \in G$ we have $f_\omega(g)=h^{-1}\cdot f_\omega(g)=f_\omega(hg)-f_\omega(h)$, which implies that $f_\omega$ is a homomorphism $H\to\Z$ and that $f_\omega(Hg)=f_\omega(H)+f_\omega(g)$ for every $g\in G$. In particular, if $f_\omega (H) = \{0\}$ then $f_\omega$ is constant on the finitely many cosets of $H$, contradicting the fact that horofunctions are unbounded.
We conclude that the image $f_\omega(H)$ is a non-trivial
subgroup of $\Z$, and thus admits a surjective homomorphism onto $\Z$.
\end{proof}

\begin{remark}
Essentially the same argument shows more generally that if $\Gamma$ is a graph of linear growth and $G<\Aut(\Gamma)$ acts transitively on the vertices of $\Gamma$ then $G$ has a finite-index subgroup admitting a surjective homomorphism onto $\Z$.
\end{remark}

\begin{proof}[Proof of Theorem \ref{thm:lin}]
Theorem \ref{thm:lin.Buse} implies that $G$ has a finite set of horofunctions.
The set of horofunctions is invariant, so in this case it contains a finite orbit.
Lemma \ref{lem:fin.orbit} therefore implies that there exists $N \lhd G$ of finite index such that $N$ admits a surjective homomorphism onto $\Z$.
Let $K \lhd N$ be the kernel of this homomorphism. 
Since $N$ is finite index in $G$ it is finitely generated of linear growth.
Since $N/K \cong \Z$, it must be that $K$ is finite.  Hence, $N$ is finite-by-$\Z$, 
which by standard methods implies that $N$ is also $\Z$-by-finite.
Thus, $G$ contains a finite-index infinite cyclic subgroup.
\end{proof}

\section*{Acknowledgements}
We are grateful to Emmanuel Breuillard, Anders Karlsson, S\'ebastien Martineau, Tom Meyerovitch, Ville Salo and an anonymous referee for helpful comments and discussions.

MT is supported by ERC grant GA617129 `GeTeMo'.  AY is supported by the Israel Science Foundation (grant no.\ 1346/15).


\begin{thebibliography}{10}

 \bibitem{vdd-wilkie}
L. van den Dries and A. J. Wilkie. An effective bound for groups of linear growth, \textit{Arch. Math.} \textbf{42} (1984), 391--396.


\bibitem{gromov}
M. Gromov. Groups of polynomial growth and expanding maps, \textit{Publ. Math. IHES} \textbf{53} (1981), 53--73.

\bibitem{hall}
P. Hall. 
On Representatives of Subsets, \textit{J. London Math. Soc.} \textbf{10} (1935), 26--30



\bibitem{imr-sei}
W. Imrich and N. Seifter. A bound for groups of linear growth, \textit{Arch. Math.} \textbf{48} (1987), 100--104.

\bibitem{justin}
J. Justin. Groupes et semi-groupes \`a croissance lin\'eaire, \textit{C.R. Acad. Sci. Paris} \textbf{273} (1971), 212--214.


\bibitem{anders}
A. Karlsson. \textit{Ergodic theorems for noncommuting random products} (lecture notes)
\url{http://www.unige.ch/math/folks/karlsson/wroclawtotal.pdf}


\bibitem{walsh}
C. Walsh.
The action of a nilpotent group on its horofunction boundary has finite orbits, 
\textit{Groups Geom. Dyn.} \textbf{5} (2011), 189--206.



\end{thebibliography}
\end{document}